\documentclass[11pt,leqno]{amsart}

\usepackage{amsmath,amsthm, amssymb, latexsym, url,amsthm,amsfonts}
\usepackage{hyperref}

\usepackage{float}

\usepackage{graphicx}
\DeclareGraphicsExtensions{.eps, .pdf,.jpg,.png}

\graphicspath{{Graphics/}}

\usepackage[margin=1.3in]{geometry}

\newtheorem{thm}{Theorem}[section]
\newtheorem{prop}[thm]{Proposition}
\newtheorem{cor}[thm]{Corollary}
\newtheorem{lem}[thm]{Lemma}

\theoremstyle{definition}
\newtheorem{defn}[thm]{Definition}

\theoremstyle{remark}

\newtheorem{remarks}[thm]{Remarks}
\newtheorem{ex}[thm]{Example}

\numberwithin{equation}{section}

\newcommand{\QQ}{\mathbb{Q}}
\newcommand{\ZZ}{\mathbb{Z}}
\newcommand{\NN}{\mathbb{N}}
\newcommand{\RR}{\mathbb{R}}

\newcommand{\RRk}{\mathbb{R}^k}
\newcommand{\ZZk}{\mathbb{Z}^k}
\newcommand{\ZZkz}{\mathbb{Z}^k\setminus\{\bfo\}}

\newcommand{\bfo}{\mathbf{0}}
\newcommand{\bff}{\mathbf{f}}
\newcommand{\bft}{\mathbf{t}}
\newcommand{\bfs}{\mathbf{s}}

\newcommand{\ckz}{C_{k,0}}
\newcommand{\cko}{C_{k,1}}
\newcommand{\cka}{C_{k,\alpha}}
\newcommand{\cki}{C_{k,i}}

\begin{document}


\title{%
A local Benford Law for a class of arithmetic sequences}

\author[Cai]{Zhaodong Cai}
\email[Zhaodong Cai]{zhcai@sas.upenn.edu}
\address{Department of Mathematics\\
University of Pennsylvania\\
Philadelphia, PA 19104-6395\\
USA}

\author[Hildebrand]{A.J. Hildebrand}
\email[A.J. Hildebrand (corresponding author)]{ajh@illinois.edu}
\address{Department of Mathematics\\
University of Illinois\\
Urbana, IL 61801\\
USA}

\author[Li]{Junxian Li}
\email[Junxian Li]{jli135@illinois.edu}
\address{Department of Mathematics\\
University of Illinois\\
Urbana, IL 61801\\
USA}

\thanks{This work is based on a research project
carried out at the \emph{Illinois Geometry Lab} in 2016. 
Some of the experimental results in this paper were generated using the
\emph{Illinois Campus Computing Cluster}, a high performance
computing platform at the University of Illinois.
}



\date{September 17, 2018}

\begin{abstract}
It is well-known that sequences such as the Fibonacci numbers and the
factorials satisfy Benford's Law; that is, leading digits in these
sequences occur with frequencies given by $P(d)=\log_{10}(1+1/d)$,
$d=1,2,\dots,9$.  In this paper, we investigate leading digit
distributions of arithmetic sequences from a local point of view.  We
call a sequence locally Benford distributed of order $k$ if,
roughly speaking, $k$-tuples of consecutive leading digits behave like
$k$ independent Benford-distributed digits. This notion refines that of a
Benford distributed sequence, and it provides a way to quantify the
extent to which the Benford distribution persists at the local level. 
Surprisingly, most sequences known to satisfy Benford's Law have 
rather poor local distribution properties. 
In our main result we establish, for a large class of arithmetic
sequences, a ``best-possible'' local Benford Law; that is, we determine
the maximal value $k$ such that the sequence is locally Benford
distributed of order $k$.  The result applies, in particular, to  sequences
of the form $\{a^n\}$, $\{a^{n^d}\}$,
and $\{n^{\beta} a^{ n^\alpha}\}$, as
well as the sequence of factorials $\{n!\}$ 
and similar iterated product sequences.
\end{abstract}

\maketitle



\section{Introduction}

\subsection{Benford's Law}
\emph{Benford's Law} refers to the phenomenon that the leading
digits in many real-world data sets 
tend to satisfy
\begin{equation}
\label{eq:benford}
P(\text{leading digit is $d$})
=\log_{10}\left(1+\frac1d\right),\quad d=1,2,\dots,9.
\end{equation}
Thus, in a data set satisfying Benford's Law, a fraction of
$\log_{10}(1+1/1)$, or around  $30.1\%$, of all numbers in the set have
leading digit $1$ in their decimal representation, a fraction of
$\log_{10}(1+1/2)\approx 17.6\%$ have leading digit $2$, and so on.

The peculiar first-digit distribution  given by \eqref{eq:benford} was
first observed in 1881 by the astronomer Simon Newcomb \cite{newcomb} 
in tables of logarithms.  It did not receive much attention until 
some fifty years later when the physicist Frank
Benford \cite{benford} compiled extensive empirical evidence
for the ubiquity of this distribution across a wide range of real-life
data sets.  In a now classic table, Benford tabulated the distribution of
leading digits in twenty different data sets, ranging from areas of
rivers to numbers in street addresses and physical constants.  Benford's
table shows good agreement with Benford's Law for most of these data
sets, and an even better agreement if all sources of data are combined
into a single data set.

In recent decades, Benford's Law has received renewed interest, in
part  because of its applications as a tool in fraud detection.
Several books on the topic have appeared in recent years  
(see, e.g., \cite{berger2015},
\cite{miller2015}, 
\cite{nigrini2012}), and close to one thousand articles have been published
(see \cite{benfordonline}).  For an overview of Benford's Law, its
applications, and  its history we refer to the papers by Raimi
\cite{raimi1976} and Hill \cite{hill1995}. An in-depth survey of 
the topic can be found in the paper by Berger and Hill \cite{berger2011}.

\subsection{Benford's Law for mathematical sequences}
From a mathematical point of view, Benford's Law is closely connected
with the theory of \emph{uniform distribution modulo $1$} \cite{kuipers}.
In 1976 Diaconis \cite{diaconis} used this connection to prove rigorously
that Benford's Law holds (in the sense of asymptotic density) for
a class of exponentially growing sequences that includes the powers of $2$,
$\{2^n\}$, the Fibonacci numbers, $\{F_n\}$, and the sequence of
factorials, $\{n!\}$.  That is, in each of these sequences, the
asymptotic frequency of leading digits is given by \eqref{eq:benford}.

In recent years, a variety of other (classes of) natural
arithmetic sequences have been shown to satisfy Benford's Law.  In
particular, in 2011 Anderson, Rolen and Stoehr \cite{anderson2011} showed
that Benford's Law holds for the partition function $p(n)$ and for the
coefficients of an infinite class of modular forms.  In 2015,  Mass\'e and
Schneider \cite{masse2015} established Benford's Law for a class of
fast-growing sequences defined by iterated product operations, including
the superfactorials, $\prod_{k=1}^n k!$, the hyperfactorials,
$\prod_{k=1}^n k^k$, and sequences of the form $2^{P(n)}$, where $P(n)$ is a
polynomial.
On the other hand, the validity of Benford's Law
for doubly exponential sequences such as $\{2^{2^n}\}$ or $\{2^{F_n}\}$
remains an open problem.


\begin{figure}[h]
\begin{center}
\includegraphics[width=.7\textwidth]{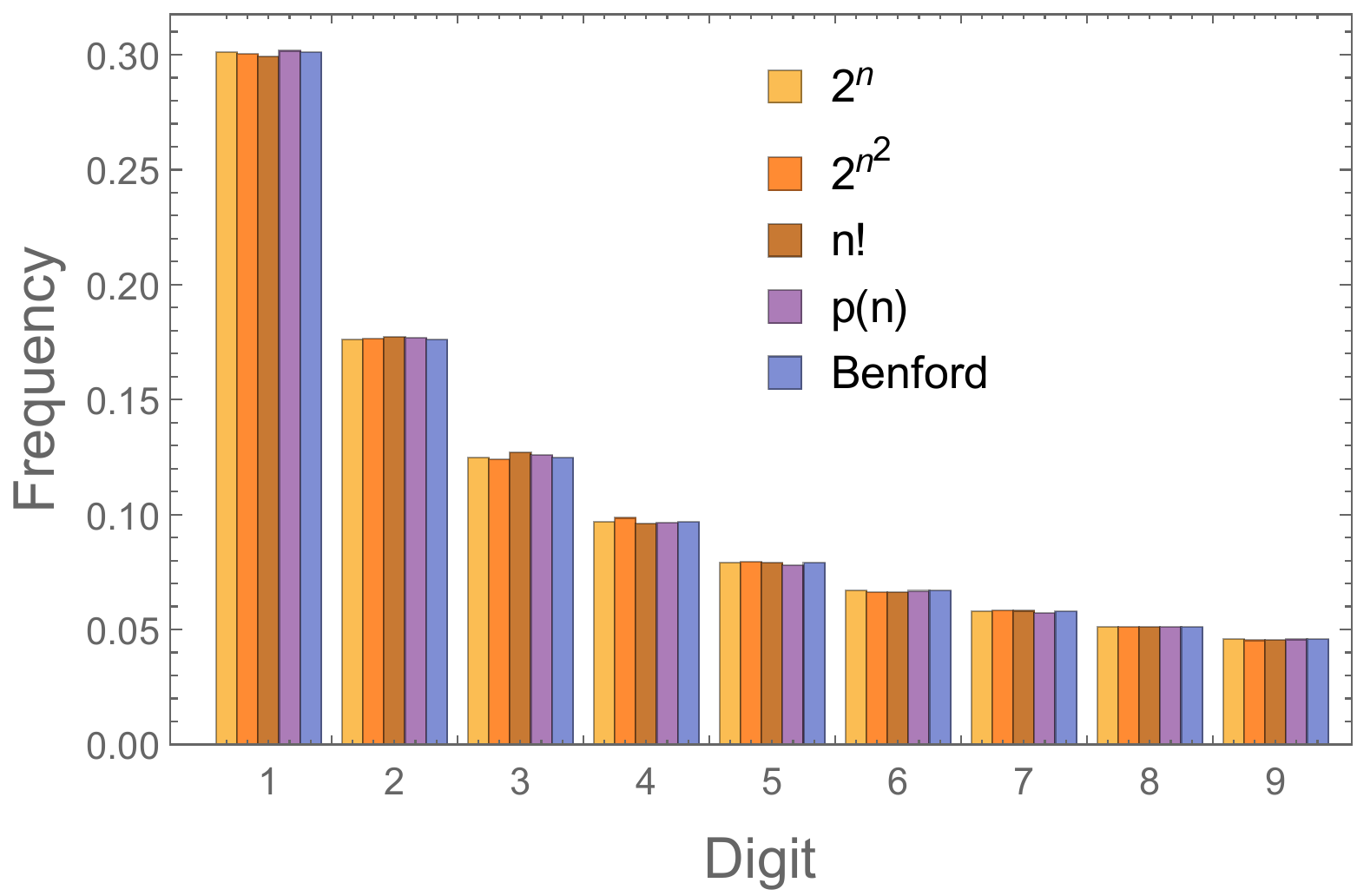}
\end{center}
\caption{Frequencies of leading digits (in base $10$) among the first
$10^5$
terms of the sequences $\{2^n\}$, $\{2^{n^2}\}$, $\{n!\}$,
and $\{p(n)\}$, 
along with the frequencies given by the Benford distribution
\eqref{eq:benford}.}
\label{fig:global}
\end{figure}


Figure \ref{fig:global} illustrates these results, showing the frequencies of
leading digits  for the sequences $\{2^n\}$, $\{2^{n^2}\}$, $\{n!\}$, and
$\{p(n)\}$ (where $p(n)$ is the partition function). 
The leading digit frequencies of all four sequences are in excellent 
agreement with the frequencies predicted by Benford's Law.


In contrast to these positive results, it has long been known 
(and is easy to see, e.g., by considering intervals of
the form $[10^k,2\cdot 10^k)$) that sequences of polynomial (or slower) rate of
growth such as $\{n\}$ or $\{n^2\}$ do not satisfy Benford's Law in the
usual asymptotic density sense.  In many of these cases, Benford's Law
can be shown to hold in some weaker form, for example, with the natural
asymptotic density replaced by other notions of density (see
\cite{masse2011} for a survey).


\subsection{Local Benford distribution}
As Figure \ref{fig:global} shows, in terms of the \emph{global} distribution
of leading digits, the four sequences $\{2^n\}$, $\{2^{n^2}\}$, $\{n!\}$, and
$\{p(n)\}$  all seem to behave in essentially the same way. This raises the
question of whether one can distinguish between such leading digit sequences
in some other way.  For example, if we are given a block of  consecutive
leading digits from each of these four sequences, as in Table
\ref{table:leadingdigits} below, can we tell, with a reasonable level of
confidence, to which sequence each block belongs?


\begin{table}[H]

\begin{center}
\begin{tabular}{|c|l|}
\hline
Sequence & Leading digits of first $50$ terms (concatenated)
\\
\hline
$\{2^n\}$&
2481361251
2481361251
2481361251
2481361251
2481371251
\\
\hline
$\{2^{n^2}\}$ &
2156365121
2271519342
5412132118
1169511474
1146399353
\\
\hline
$\{n!\}$ &
1262175433
3468123612
5126141382
8282131528
3162152163
\\
\hline
$\{p(n)\}$ &
1235711234
5711122346
7111123345
6811112233
4567811112
\\
\hline
\end{tabular}
\end{center}
\caption{Leading digits (in base $10$) of the first $50$
terms of the sequences $\{2^n\}$, $\{2^{n^2}\}$, $\{n!\}$, 
and $\{p(n)\}$.
}
\label{table:leadingdigits}
\end{table}

All four sequences in this table are  
known to satisfy Benford's Law, so in terms of the \emph{global}
distribution of leading digits they behave in roughly the same way. 
This behavior is already evident in the limited data shown in Table
\ref{table:leadingdigits}.  For example, among the first $50$ terms of the
sequence $\{2^n\}$ exactly $15$ have leading digit $1$, while the digit $1$
counts for the other three sequences are $15$, $14$, and $18$, respectively. 
These counts are close to the counts predicted by Benford's law, namely
$50\cdot \log_{10}2\approx 15.05\dots$.

A closer examination of Table \ref{table:leadingdigits} reveals significant
differences at the \emph{local} level: The leading digits of $\{2^n\}$ exhibit
an almost periodic behavior with a strong (and obvious) correlation between
consecutive terms, while the sequence $\{p(n)\}$ shows a noticeable tendency
of digits to repeat themselves.  On the other hand, the leading digits of the
sequences $\{2^{n^2}\}$ and $\{n!\}$ appear to behave more ``randomly'',
though it is not clear to what extent this randomness persists at the local
level.  Is one of the latter two sequences more ``random'' in some sense than
the other?  

In this paper we seek to answer questions of this type by studying
the leading digit distribution of arithmetic sequences from a \emph{local}
point of view. More precisely, we will focus on the distribution of $k$-tuples
of leading digits of consecutive terms in a sequence, and the question of when
this distribution is asymptotically the same as that of $k$ independent
Benford distributed random variables.


When viewed from such a local perspective, striking differences between 
sequences can emerge.  This is illustrated in Figure \ref{fig:local}, which
shows the frequencies of (selected) pairs of leading digits for the same four
sequences that we considered in Figure \ref{fig:global}.  In stark contrast to
the single digit frequencies shown in Figure \ref{fig:global}, the frequencies
of pairs of leading digits  vary widely from sequence to sequence.



\begin{figure}[H]
\begin{center}
\includegraphics[width=.6\textwidth]{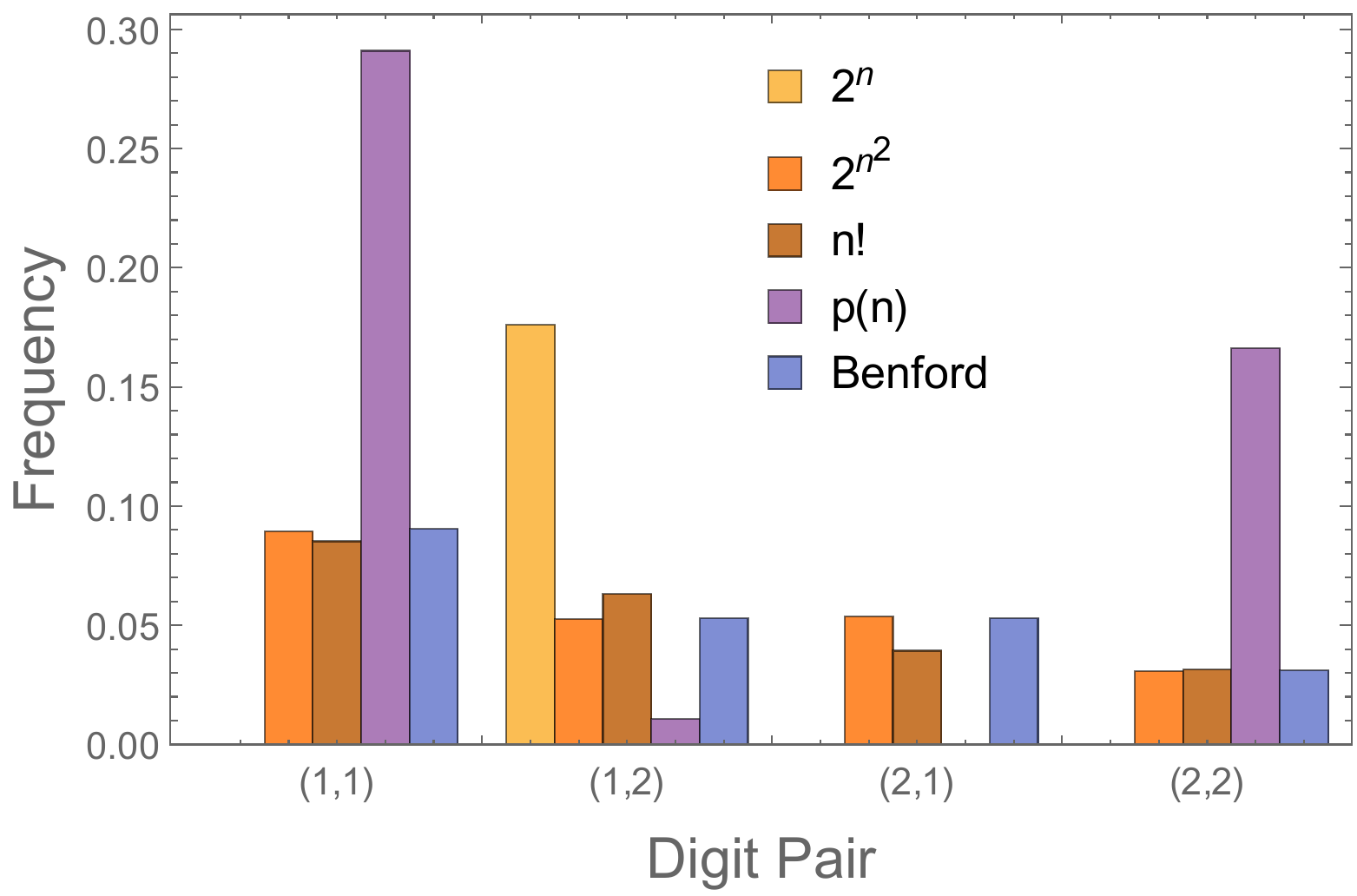}
\end{center}

\caption{Frequencies
of selected pairs $(d_1,d_2)$ of 
leading digits among the first $10^5$
terms of the sequences $\{2^n\}$, $\{2^{n^2}\}$, $\{n!\}$,
and $\{p(n)\}$, along with the predicted frequencies, given by 
$P(d_1)P(d_2)=\log_{10}(1+1/d_1)\log_{10}(1+1/d_2)$.
(Note that some of bars in the chart have height $0$ as the
corresponding frequencies are $0$.)} 
\label{fig:local}
\end{figure}

Figure \ref{fig:local} suggests that for the sequence $\{2^{n^2}\}$ 
(but not for any of the other sequences in this figure) pairs of leading
digits are distributed like independent Benford distributed random
variables.  Large scale computations of this sequence, shown in Table
\ref{table:twonsquared}, provide compelling numerical evidence for
this behavior.

\begin{table}[H]

\begin{center}
\begin{tabular}{|c|c|c|c|c|}
\hline
$N$& $(1,1)$ & $(1,2)$ & $(2,1)$ & $(2,2)$ 
\\
\hline
$10^4$ &
0.0888777&
0.054454&
0.0552592&
0.0312028

\\
\hline
$10^5$ &
0.0894678&
0.0528286&
0.053985&
0.0306747

\\
\hline
$10^6$ &
0.0906688&
0.0527907&
0.0528264&
0.0308901

\\
\hline
$10^7$ &
0.0906353&
0.052968&
0.0529541&
0.0309502

\\
\hline
$10^8$ &
0.0906542&
0.0529921&
0.0529683&
0.0310264

\\
\hline
$10^9$&
0.0906257&
0.0530009&
0.0530023&
0.0310054
\\
\hline
\hline
Benford  &
0.0906191&
0.0530088&
0.0530088&
0.0310081
\\
\hline
\end{tabular}
\end{center}
\caption{Frequencies of selected pairs $(d_1,d_2)$ of 
leading digits among the first $N$
terms ($N=10^k$, $k=4,5,\dots,9$) of the sequence $\{2^{n^2}\}$, 
along with the predicted frequencies, given by
$P(d_1)P(d_2)=\log_{10}(1+1/d_1)\log_{10}(1+1/d_2)$.}
\label{table:twonsquared}
\end{table}

The results we will prove in this paper confirm the behavior 
suggested by Figure \ref{fig:local}
and Table \ref{table:twonsquared}.
We will show 
that, among the four sequences shown in Figures \ref{fig:global} and
\ref{fig:local}, only the sequence $\{2^{n^2}\}$ has the property that
pairs of leading digits are distributed like independent Benford distributed
variables.  In this sense the sequence $\{2^{n^2}\}$ is the ``most
random'' among these four sequences.

\subsection{Summary of results and outline of paper}
We call a sequence \emph{locally Benford distributed of
order $k$} if, roughly speaking,  $k$-tuples of leading digits of
consecutive terms in this sequence  are distributed like independent
Benford-distributed digits (see Definition \ref{def:localBenford} below
for a precise statement).  This notion refines that of a Benford
distributed sequence, which corresponds to the case $k=1$,  and it
provides a way to quantify the extent to which a Benford distributed
sequence retains this property at the local level. 

Given a Benford-distributed sequence (or, equivalently, a sequence that is
locally Benford of order $1$), one can ask for the maximal value 
$k^*$ (if it exists) such that the sequence is locally Benford of order
$k^*$.  If such a value $k^*$ exists, we call $k^*$ the \emph{maximal
local Benford order} of the sequence; otherwise we say that the sequence
has infinite maximal local Benford order.

Our main result, Theorem \ref{thm:main1}, determines this maximal local
Benford order for a large class of arithmetic sequences that includes 
the main classes of sequences known to satisfy Benford's Law. 
The result applies, in particular, to 
sequences of the form $\{a^n\}$, $\{a^{n^d}\}$, 
and $\{n^{\beta} a^{ n^\alpha}\}$, as well as 
the factorials, the superfactorials, $\prod_{k=1}^nk!$,
and similar iterated product sequences,
and it allows us to classify these sequences according to
their maximal local Benford order. 
For example, we will show 
(see Example \ref{ex:cor2}(2)(4), Corollary \ref{thm:cor1}, 
and Example \ref{ex:iteratedprod}(1) along with Theorem \ref{thm:main1}(ii))
that the sequences $\{2^n\}$, $\{n!\}$, and $\{p(n)\}$ 
have maximal local Benford order $1$ (so that, in particular,
pairs of consecutive leading digits are not independent), while  $\{2^{n^2}\}$
has maximal local Benford order $2$ (so that pairs of consecutive leading
digits are independent, while triples are not).  This answers the above
question about the degree of ``local randomness'' in the four sequences shown in
Table \ref{table:leadingdigits}: The sequence $\{2^{n^2}\}$ is the most
``random'' of these sequences, in the sense of having the largest maximal
Benford order. 

The sequences covered by Theorem \ref{thm:main1} 
all have finite maximal Benford order.
To complement this result, we show in Theorem \ref{thm:main2} the existence of
sequences with infinite maximal local Benford order.  Specifically, we
consider doubly exponential sequences of the form $\{a^{\theta^n}\}$, where
$a>1$. Using metric results from the theory of uniform distribution modulo
$1$, we show that,  for almost all real numbers $\theta>1$, such a sequence
has infinite maximal local Benford order.  We also show that when
$\theta$ is an \emph{algebraic} number, the sequence $\{a^{\theta^n}\}$ 
has finite maximal local Benford order, given by the degree of $\theta$ over
$\QQ$.  The latter result applies, in particular, to the sequences
$\{2^{2^n}\}$ and $\{2^{F_n}\}$, which are conjectured (but not known) to
be Benford distributed. It shows that the maximal local Benford order of
these sequences is at most $1$ in the case of $\{2^{2^n}\}$, and $2$ in the
case of $\{2^{F_n}\}$.

\medskip

The remainder of this paper is organized as follows. In Section
\ref{sec:notations}, we introduce some notation and 
state our main results, Theorems
\ref{thm:main1} and \ref{thm:main2},
and some consequences and corollaries of these results.
In Section \ref{sec:uniform-distribution}, 
we introduce some background from the theory of
uniform distribution modulo $1$ and we 
state our key tool, Proposition \ref{prop:key}, 
a result on the uniform distribution modulo
$1$ of $k$-tuples $(f(n),\dots,f(n+k-1))$, for certain classes of
functions $f(n)$. We conclude this section by deducing  
Theorem \ref{thm:main1} from Proposition \ref{prop:key}. 
Section \ref{sec:proof-prop-key} is devoted to the proof  of Proposition
\ref{prop:key}. Section \ref{sec:proof-corollaries} contains the proof  
of the corollaries of Theorem \ref{thm:main1}. 
Theorem \ref{thm:main2} is proved in Section  \ref{sec:proof-thm-main2}.
The final section, Section \ref{sec:concluding-remarks}, contains some
remarks on related results and possible extensions and generalizations 
of our results.


\section{Notation and statement of results}
\label{sec:notations}

\subsection{Notational conventions}

Following the lead of most of the recent literature on 
Benford's Law in a mathematical context, we consider leading digits 
with respect to expansions in a 
general base $b$, where $b$ is an integer $\ge 2$.  The base $b$ analog
of the  Benford distribution \eqref{eq:benford} is given by
\begin{equation} \label{eq:benfordBaseb}
P_b(d)=\log_{b}\left(1+\frac1d\right),\quad d=1,2,\dots,b-1.
\end{equation} 
The notations, definitions, and results  we will introduce are understood
to hold in this general setting. The dependence on $b$ may not
be explicitly stated if it is clear from the context.

We let $\{a_n\}$ denote a sequence of positive real numbers, indexed by
the natural numbers $n=1,2,\dots$.  When convenient, we will use
functional instead of subscript notation, and write a sequence as
$\{f(n)\}$. 

Given a real number $x$, we denote by $\lfloor x\rfloor$ (resp. $\lceil x
\rceil$) the floor (resp. ceiling) of $x$, and we let 
$\{x\}=x-\lfloor x\rfloor$ denote the fractional part of $x$.  (The curly
brace notation is also used to denote sequences, but the meaning will
always be clear from the context.)  

We denote by $\log_bx$ the logarithm of $x$ in base $b$; that is, 
$\log_b x=(\log x)/(\log b)$. 

We use the asymptotic notations $\sim$, $o(\dots)$, and
$O(\dots)$, in the usual sense: $f(n)\sim g(n)$ means $\lim_{n\to\infty}
f(n)/g(n)=1$; $f(n)=o(g(n))$ means $\lim_{n\to\infty} f(n)/g(n)=0$; and
the notation $f(n)=O(g(n))$ means that there exists a constant $C$
(independent of $n$, but possibly depending on other parameters) such
that $|f(n)|\le C |g(n)|$ holds for all $n$.

Vectors and vector-valued functions are denoted by boldface symbols, and
their components are indicated by subscripts, with indices starting at
$0$.  For example,  $\bff(n)=(f_0(n),\dots,f_{k-1}(n))$,
$\bft=(t_0,\dots,t_{k-1})$.  We denote by $\bfo$ the zero vector.

\subsection{Local Benford distribution}

Given a positive real number $x$ and an integer base $b\ge 2$, we let
$d_b(x)$ denote the leading (i.e., most significant) digit of $x$ when
expressed in base $b$.  We then have, for any $d\in\{1,2,\dots,b-1\}$, 
\begin{align*}
d_b(x) = d
& \Longleftrightarrow 
d \cdot b^k\le x<(d+1) b^k \text{ for some $k\in\ZZ$ }
\\
& \Longleftrightarrow  
\log_b d \le   \{\log_{b} x\}< \log_b(d+1).
\end{align*}
This equivalence relates the distribution of leading digits in base
$b$ of a set of numbers $x$ to  that of the fractional parts $\{\log_b x\}$. 
In particular, if these fractional parts are uniformly distributed in the
interval $[0,1]$, then  for each $d$ the proportion of numbers $x$ with
leading digit $d$ will be $\log_b(d+1)-\log_b d = \log_b (1+1/d)$, i.e.,  
the probability $P_b(d)$ given in \eqref{eq:benfordBaseb}.
This motivates the following definition:

\begin{defn}[Local Benford distribution]
\label{def:localBenford}
Let $b$ be an integer base $\ge 2$, and let $\{a_n\}$ be a sequence of 
positive real numbers. 
\begin{itemize}
\item[(i)] The sequence $\{a_n\}$ is called \emph{Benford distributed with
respect to base $b$} if 
\begin{equation}
\label{eq:defBenford}
\lim_{N\to\infty} \# \frac1N\{n\le N: \{\log_b a_n\} \le \alpha\}=\alpha 
\quad (0\le \alpha\le 1).
\end{equation}

\item[(ii)] Let $k$ be a positive integer. 
The sequence $\{a_n\}$ is called \emph{locally Benford distributed of
order $k$ with respect to base $b$} if 
\begin{align}
\label{eq:defLocalBenford}
\lim_{N\to\infty} &\#\frac1N\{n\le N: \{\log_b a_{n+i}\}
\le \alpha_i \ (i=0,1,\dots,k-1)\}
\\
\notag
&\qquad
= \alpha_0\alpha_1\dots \alpha_{k-1}
\quad (0\le \alpha_0,\alpha_1,\dots,\alpha_{k-1}\le 1).
\end{align}
\end{itemize}
\end{defn}

\begin{remarks}
\mbox{}
(1) An alternative way to define a Benford distributed sequence 
would be to require the leading digits to have the asymptotic frequencies
given by \eqref{eq:benfordBaseb}: 
\begin{equation}
\label{eq:defBenfordalt}
\lim_{N\to\infty} \# \frac1N\{n\le N: d_b(a_n)=d\}
=\log_b\left(1+\frac1d\right)
\quad (d=1,\dots,b-1).
\end{equation}
This amounts to restricting $\alpha$ in the given definition,
\eqref{eq:defBenford}, to the discrete values $\alpha=\log_b (d+1)$,
$d=1,\dots,b-1$, and thus yields a slightly weaker property.\footnote{%
As shown by Diaconis \cite{diaconis}, the definition 
\eqref{eq:defBenford} is equivalent to the property that, for any 
positive integer $D$, the asymptotic frequency of terms $a_n$ whose
``leading digit block'' is given by the base $b$ expansion of $D$, 
is equal to $\log_b(1+1/D)$.}
Similarly, the notion of local Benford distribution could have been
defined by the relation 
\begin{align}
\label{eq:defLocalBenfordalt}
\lim_{N\to\infty} &\# \frac1N\{n\le N: d_b(a_{n+i})
=d_i \  (i=0,1,\dots,k-1)\} 
\\
\notag
&\qquad =\prod_{i=0}^{b-1}\log_b\left(1+\frac1{d_i}\right)
\quad (d_i=1,\dots,b-1), 
\end{align}
which would yield a slightly weaker property.

We decided to adopt the stronger definition,  
\eqref{eq:defBenford}, as the basis for our concept of local Benford
distribution since this is the definition most commonly used in 
the recent literature on the subject.

(2) In the case $k=1$, the definition of local Benford distribution, 
\eqref{eq:defLocalBenford}, reduces to that of the  ordinary Benford
distribution \eqref{eq:defBenford}.

(3) It is immediate from the definition that local Benford distribution
of order $k$ implies local Benford distribution of any smaller order
$k'<k$, and in particular implies  
Benford distribution in the sense of \eqref{eq:defBenford}.
As pointed out by the referee, a more general version of this
observation holds: If
a sequence $\{a_n\}$ is locally Benford distributed of order $k$, then for any
integers $0\le i_1<i_2<\dots<i_h\le k$, the tuple
$(a_{n+i_1},\dots,a_{n+i_h})$ behaves like a tuple of independent Benford
distributed random variables.

\end{remarks}

\begin{defn}[Maximal local Benford order]
\label{def:maxorder}
Let $b$ be an integer base $\ge 2$ and let $\{a_n\}$ be a sequence of
positive real numbers that is Benford distributed with respect to base $b$
(and hence also locally Benford distributed of order at least $1$). 
\begin{itemize}
\item[(i)] 
If there exists a maximal integer $k^*$ such that $\{a_n\}$ 
is locally Benford distributed of order $k^*$, then we call $k^*$ 
the \emph{maximal local Benford order} of the sequence $\{a_n\}$ 
with respect to base $b$.
\item[(ii)] 
If no such integer $k^*$ exists (i.e., if $\{a_n\}$ 
is locally Benford distributed of any order), we say that $\{a_n\}$ has 
\emph{infinite maximal local Benford order} with respect to base $b$.
\end{itemize}
\end{defn}

\subsection{The classes $\cki$}

Let $\Delta$ be the difference operator (or discrete derivative)
defined by
\begin{equation}
\label{eq:defDelta}
\Delta f(n)=f(n+1)-f(n), 
\end{equation}
and let $\Delta^k$ denote the $k$-th iterate of this operator, so that
\begin{equation}
\label{eq:defDeltak}
\Delta^1 f(n)=\Delta f(n), \quad
\Delta^{k+1} f(n)= \Delta (\Delta^k f(n))\quad
(k=1,2,\dots).
\end{equation}

We classify sequences $\{f(n)\}$ into classes $C_{k,i}$ according
their asymptotic behavior under this iterated difference operator:

\begin{defn}[Classes $\cki$]
\label{def:ck}
Let $\{f(n)\}$ be a sequence of real numbers and let $k$ be a positive
integer.  We say that $\{f(n)\}$ is of 
\begin{itemize}
\item[(i)] class $\ckz$ if for some $\theta\not\in\QQ$, 
\[
\lim_{n\to\infty} \Delta^k f(n)=\theta;
\]

\item[(ii)] class $\cka$, where $0<\alpha<1$, 
if for some $\lambda\not=0$,
\[
\lim_{n\to\infty} n^\alpha \Delta^k f(n)=\lambda;
\]

\item[(iii)] class $\cko$, if for some $\lambda\not=0$,
\[
\lim_{n\to\infty} n \Delta^{k+1} f(n)=\lambda.
\]
\end{itemize}
We denote by $\cki$ any of the three classes (i)--(iii). 
\end{defn}

Roughly speaking, class $\ckz$ covers functions of growth 
proportional to $n^k$ and with irrational leading coefficient, 
class $\cka$ covers functions of growth proportional to $n^{\beta}$, 
where $\beta>0$ is \emph{not} an integer, 
while class $\cko$ covers functions of growth proportional 
to $n^k\log n$.

\begin{remarks}
\label{rem:ck}
\mbox{}
(1) Note that the definition of the class $\cko$ involves the $(k+1)$st
iterate of the $\Delta$ operator, whereas the other two classes involve
the $k$th iterate. This convention will allow us to state our results in
a uniform manner for all three classes.  

(2) It is not hard to see that a sequence $\{f(n)\}$ can belong to at most
one of the classes $\cki$; that is, both $k$ and the class subscript $i$
are uniquely determined by the sequence $\{f(n)\}$. For example, suppose 
$f$ is of class $\ckz$ for some $k\in\NN$. Then
\[
\Delta^{k+1}f(n)=\Delta^kf(n+1)-\Delta^kf(n)\to \theta-\theta=0, 
\]
and inductively we get $\Delta^{k+i}f(n)\to0$ for all $i\ge 1$. 
Thus $f$ cannot be of class $C_{k',0}$ for some $k'>k$.
Interchanging the roles of $k'$ and $k$, we see that
$f$ also cannot be of class $C_{k',0}$ for some $k'<k$. Hence, a 
sequence $\{f(n)\}$ can be in class $\ckz$ for a most one value of $k$.
Using similar reasoning one can show that $k$ is uniquely determined for
sequences in the other two classes, $\cka$ and $\cko$, and that a
sequence can belong to at most one such class.

(3) It is clear from the definition that $\{\Delta f(n)\}$ belongs to
$C_{k,i}$ if and only if $\{f(n)\}$  belongs to
$C_{k+1,i}$.   More generally, for any positive integer $h$ 
the sequence $\{\Delta^h f(n)\}$ belongs to $C_{k,i}$ if and only if 
the sequence $\{f(n)\}$ belongs to $C_{k+h,i}$.

(4) A sufficient condition for $f$ to be in some class $\cki$ is that $f$
is a function defined on $[1,\infty)$ satisfying the continuous analog of
the given condition, i.e., with the $k$-th discrete derivative $\Delta^k
f(n)$ replaced by the ordinary $k$-th derivative $f^{(k)}(x)$.  
To see this, note that if $f$ has continuous derivatives up to order $k$,
then $\Delta^k f$ can be represented as an integral over the $k$-th
derivative of $f$  of the form $(*)$ $\Delta^k f(n)=\int_{0}^k f^{(k)}(n+t)
\phi_k(t)dt$, where $\phi_k(t)$ is a nonnegative kernel function supported on
$[0,k]$ that integrates to $1$.  The identity $(*)$ can be proved by
induction. 
\end{remarks}

\medskip

Using the last of these remarks easily yields a large class of examples 
of functions belonging to one of the classes $\cki$: 

\begin{ex}
\label{ex:ck}
\mbox{}
\begin{itemize}
\item[(1)]
For any nonconstant polynomial $P(x)$ with 
irrational leading coefficient, $\{P(n)\}$ is of class $\ckz$ with
$k=\deg(P)$.
In particular, $\{\theta n^k\}$ is of class $\ckz$ whenever $\theta$ is
irrational.

\item[(2)]
For any nonconstant polynomial $P(x)$, the sequence $\{P(n)\log n\}$ 
is of class $\cko$ with $k=\deg(P)$. 
In particular,  $\{n^k\log n\}$ is of class $\cko$.

\item[(3)]
For any positive real number $\beta$ that is not an integer,   
the sequence $\{n^\beta\}$ is of class $\cka$ with
$k=\lceil\beta\rceil$ and $\alpha=k-\beta$.

\end{itemize}

\end{ex}

\subsection{Iterated product sequences}
Given a sequence $\{a_n\}$, we define 
the \emph{iterated product sequences} $\{a_n^{(h)}\}$ 
by (cf. \cite{masse2015}) 
\begin{equation}
\label{eq:iteratedprod}
a_n^{(1)}= a_n, \quad a_n^{(h+1)} = \prod_{m=1}^n a_m^{(h)}\quad
(h=1,2,\dots).
\end{equation}

\pagebreak[3]

\begin{ex}
\label{ex:iteratedprod}
\mbox{}
\begin{itemize}
\item[(1)]
If $a_n=n!$, then the numbers $n!^{(2)}=\prod_{m=1}^n m!$ are the 
superfactorials, and the numbers $n!^{(h)}$ are generalized
superfactorials, obtained by iterating the product operation. 

\item[(2)]
If $a_n=a^{P(n)}$, where $P$ is a polynomial of degree $d$, then
$a_n^{(h)}=a^{P_h(n)}$, where $P_h(n)$ is a polynomial of degree $d+h-1$, 
defined inductively by $P_1(n)=P(n)$ and $P_{h+1}(n)=\sum_{m=1}^n
P_h(m)$.
\end{itemize}
\end{ex}

\subsection{Statement of results}
With the above definitions, we are ready to state our main result. 

\begin{thm}[Main Theorem]
\label{thm:main1}
Let $b$ be an integer base $\ge 2$, let $\{a_n\}$ be a sequence of 
positive real numbers, and suppose the sequence $\{\log_b a_n\}$ belongs
to one of   the classes $\cki$ in Definition \ref{def:ck}.
Then:
\begin{itemize}
\item[(i)] The sequence $\{a_n\}$ 
has maximal local Benford order $k$ with respect to base $b$; 
that is, $\{a_n\}$ is locally Benford distributed of order $k$, 
but not of any higher order.
\item[(ii)] For any positive integer $h$, 
the iterated product sequences $\{a_n^{(h)}\}$ has 
maximal local Benford order $k+h-1$ with respect to base $b$. 
\end{itemize}
\end{thm}

The conditions on $\{a_n\}$ in this theorem are general enough to 
cover most of the classes of functions previously considered in the
literature.  In the following corollaries we give some special cases and
consequences of this result.

Our first corollary is motivated by applications to the partition
function and coefficients of modular forms (see \cite{anderson2011}).  

\begin{cor}
\label{thm:cor1}
Suppose 
\begin{equation}
\label{eq:cor1}
a_n\sim \lambda n^{\gamma} e^{c n^\beta},
\end{equation}
where $\lambda,\gamma,c,\beta$ are constants with 
$\lambda>0,c>0,\beta>0$, and  
$\beta$ not an integer.  Then 
$\{a_n\}$ has maximal local Benford order $\lceil \beta\rceil$ with
respect to any base $b\ge 2$.
\end{cor}

In particular, since the partition function $p(n)$ satisfies a relation
of the form \eqref{eq:cor1} with $\beta =1/2$, the corollary shows that
$\{p(n)\}$ has maximal local Benford order $\lceil 1/2\rceil = 1$ 
with respect to any base
$b\ge 2$. Thus, while the leading digits of the partition function are
Benford distributed (as previously shown in \cite{anderson2011}), 
pairs of leading digits of consecutive terms do \emph{not} behave like
independent Benford distributed digits.

The next corollary concerns a very wide class of functions introduced in
\cite{masse2015} (see Definition 3.4 and Theorem 3.10 of \cite{kuipers})
that includes, for
example, sequences of geometric growth such as $\{2^n\}$ and the
Fibonacci sequence $\{F_n\}$, and sequences of ``super-geometric'' growth
such as $\{a^{n^d}\}$, as well as the sequence of factorials   
and similar functions such as $\{n^n\}$. 

\begin{cor}
\label{thm:cor2}
Let $b$ be an integer base $\ge 2$ and suppose 
\begin{equation}
\label{eq:cor2}
a_n\sim \lambda n^{P(n)} b^{Q(n)},
\end{equation}
where $P$ and $Q$ are polynomials and $\lambda>0$.
\begin{itemize}
\item[(i)] If $\deg(P) < \deg(Q)$ and $Q$ has irrational leading
coefficient, then $\{a_n\}$ has maximal local Benford order $k=\deg(Q)$
with respect to base $b$.

\item[(ii)] If $\deg(P)\ge \deg(Q)$ and $P$ is nonconstant, then 
$\{a_n\}$ has maximal local Benford order $k=\deg(P)$ with respect to base
$b$.
\end{itemize}
\end{cor}

We mention some particular cases of this result:
\begin{ex}
\label{ex:cor2}
\mbox{}
\begin{enumerate}
\item[(1)] The sequence
$\{n^n\}$ satisfies \eqref{eq:cor2} with $\lambda=1$, $P(n)=n$ and
$Q(n)=0$, so by part (ii) of the corollary this sequence has 
maximal local Benford order $1$.  

\item[(2)]  By Stirling's formula we have $n!\sim \sqrt{2\pi} n^{n+1/2}
e^{-n}$, so the factorial sequence $\{n!\}$ satisfies \eqref{eq:cor2}
with $P(n)=n+1/2$ and $Q(n)=-(1/\log b) n$. 
Thus, part (ii) of the corollary applies again and yields that $\{n!\}$
has maximal local Benford order $1$.  

\item[(3)] 
By Binet's Formula, we have $F_n\sim (1/\sqrt{5})\Phi^n$, where
$\Phi=(\sqrt{5}+1)/2$, so the Fibonacci sequence $\{F_n\}$ satisfies
\eqref{eq:cor2} with $P(n)=0$ and $Q(n)=(\log_b \Phi) n$.
The leading coefficient of $Q$, $\log_b (\sqrt{5}+1)/2$, is irrational
for any integer $b\ge 2$, so by part (i) of the corollary $\{F_n\}$
has maximal local Benford order $1$ with respect to any base $b\ge2$.

\item[(4)]
The sequence $\{2^{n^d}\}$, where $d$ is a positive integer, satisfies 
\eqref{eq:cor1} with $P(n)=0$ and $Q(n)=(\log_b2) n^d$.
By part (i) of the corollary it follows that this sequence has 
maximal local Benford order $d$ with respect to any base $b$ such that
$\log_b 2$  is irrational, i.e., any base $b$ that is not a power of $2$.

\end{enumerate}
\end{ex}

The above special cases include the sequences
in  Figures \ref{fig:global} and \ref{fig:local} 
Table \ref{table:leadingdigits}, and they allow us to resolve 
the question on the degree of ``local randomness'' in these sequences we
had posed in the introduction: 
By Corollary \ref{thm:cor1} the sequence $\{p(n)\}$ has maximal local Benford
order $1$.  By Example \ref{ex:cor2}, the sequences $\{n!\}$ and
$\{2^n\}$ have maximal 
local Benford order $1$, while the sequence  $\{2^{n^2}\}$
has maximal local Benford order $2$.  Thus, pairs of leading digits of consecutive
terms of $\{2^{n^2}\}$ behave like independent Benford-distributed random
variables, while this is not the case for the sequences $\{2^n\}$, 
$\{n!\}$, and $\{p(n)\}$. In this sense, the sequence $\{2^{n^2}\}$ is   
the most ``random'' of the four sequences.

\begin{cor}
\label{thm:cor3}
For any positive integer $n$, the generalized superfactorial
sequence (see Definition \ref{eq:iteratedprod}), 
$\{n!^{(h)}\}$, 
has maximal local Benford order $h$ with
respect to any base $b\ge 2$.
\end{cor}

In all of the above examples the sequence has finite maximal local
Benford order. One can ask if there exist sequences that have infinite
local Benford order, i.e., sequences that are locally Benford distributed
for any order $k$.  The above results suggest that the order of local
Benford distribution is closely related to the rate of growth of 
the sequence $\log_b a_n$. For example, if $\log_b a_n$ is a polynomial 
of degree $d$ with irrational leading coefficient, then $\{a_n\}$ 
has maximal local Benford order $d$.  Hence one might expect that 
sequences for which $\log_b a_n$ grows at exponential rate
``typically'' will have infinite maximal local Benford order.  The following
result confirms this by showing that, in some sense, almost all sequences
for which $\log_b a_n$ grows at an exponential rate have infinite local
Benford order. 

\begin{thm}[Local Benford order of doubly exponential sequences]
\label{thm:main2}
Let $a>1$ be a real number.
\begin{itemize}
\item[(i)]
For almost all real numbers  $\theta>1$ the sequence $\{a^{\theta^n}\}$
has infinite maximal local Benford order with respect to any base $b\ge 2$.
\item[(ii)] 
If $\theta>1$ is an algebraic number of degree $k$,
then, with respect to any base $b$, 
the maximal local Benford order of the sequence $\{a^{\theta^n}\}$ 
is at most $k$.
\end{itemize}
\end{thm}

\begin{ex}
\label{ex:main2}
\mbox{}
\begin{itemize}
\item[(1)]
The sequence $\{2^{2^n}\}$ satisfies the conditions of Theorem
\ref{thm:main2}(ii) with $\theta=2$. Since $2$ is algebraic
of degree $1$, the theorem shows that
the sequence cannot be locally Benford
distributed of order $2$ (or greater).   (Whether
the sequence is locally Benford distributed of order $1$, i.e., whether
it satisfies Benford's Law, remains an open question.)

\item[(2)]
By Binet's Formula, we have
\begin{equation}
\label{eq:2fn}
2^{F_n} \sim 2^{(1/\sqrt{5})\Phi^n}, 
\end{equation}
where $F_n$ denotes the $n$-th Fibonacci number and $\Phi=(\sqrt{5}+1)/2$. 
Since $\Phi$ is algebraic of degree $2$, the theorem can 
be applied to the sequence on the right of \eqref{eq:2fn}, and it shows that   
this sequence cannot be locally Benford distributed of order $3$ (or
greater).  In view of the relation \eqref{eq:2fn}
the same is true for the sequence $\{2^{F_n}\}$
(cf. Lemma \ref{lem:asympBenford} below). 
\end{itemize}
\end{ex}



\section{Uniform distribution modulo 1 and the Key Proposition}
\label{sec:uniform-distribution}

In this section, we introduce some key concepts and results from the
theory of uniform distribution modulo $1$, and we use these to reduce
Theorem \ref{thm:main1} to a statement about uniform distribution modulo
$1$, Proposition \ref{prop:key} below. The proposition will be proved in the
next section. 

\subsection{Uniform distribution modulo $1$ in $\RRk$}

We recall the standard definition of uniform distribution modulo $1$ of
sequences of real numbers, and its higher-dimensional analog; see, for
example,  Definitions 1.1 and 6.1 in Chapter 1 of 
\cite{kuipers}\footnote{In
\cite{kuipers} these definitions are given in a slightly different,
though equivalent, form, with the one-sided constraints $\{f_i(n)\}\le \alpha_i$ 
replaced by two-sided constraints $\beta_i\le \{f_i(n)\}<\gamma_i$, 
where $0\le \beta_i<\gamma_i\le 1$. 
The equivalence of the two versions is easily seen by taking linear
combinations of the quantities in \eqref{eq:defudmod1RRk} with 
$\alpha_i\in\{\beta_i,\gamma_i\}$.}.

\pagebreak[3]

\begin{defn}[Uniform distribution modulo $1$ in $\RRk$]
\label{def:udmod1}
\mbox{}
\begin{itemize}
\item[(i)] A sequence $\{f(n)\}$ of real numbers is 
said to be \emph{uniformly distributed modulo $1$} if 
\begin{equation}
\label{eq:defudmod1}
\lim_{N\to\infty} \frac1N \#\{n\le N: \{f(n)\} \le \alpha\}=\alpha
\quad (0\le \alpha\le 1).
\end{equation}
(Recall that $\{x\}$ denotes the fractional part of $x$.)
\item[(ii)]
A sequence $\{\bff(n)\}=\{(f_0(n),\dots,f_{k-1}(n))\}$ in $\RRk$
is said to be \emph{uniformly distributed modulo $1$ in $\RRk$} if 
\begin{align}
\label{eq:defudmod1RRk}
&\lim_{N\to\infty} \frac1N\# \{n\le N: \{f_i(n)\}\le \alpha_i \
(i=0,1,\dots,k-1)\} 
\\
\notag
&\qquad
= \alpha_0\alpha_1\dots \alpha_{k-1}
\quad ((\alpha_0,\alpha_1\dots,\alpha_{k-1})\in [0,1]^k).
\end{align}
\end{itemize}
\end{defn}

A key result in uniform distribution modulo $1$ is \emph{Weyl's
Criterion}, which we will state in the following form; see 
Theorems 2.1, 6.2, and 6.3 in Chapter 1 of \cite{kuipers}.

\begin{lem}[Weyl's Criterion in $\RRk$]
\label{lem:weyl}
\mbox{}
\begin{itemize}
\item[(i)] A sequence $\{f(n)\}$ of real numbers is 
uniformly distributed modulo $1$ if and only if 
\begin{equation}
\label{eq:weyl}
\lim_{N\to\infty} \frac1N\sum_{n=1}^N e^{2\pi i t f(n)} =0
\quad (t\in \ZZ\setminus\{0\}).
\end{equation}
\item[(ii)]
A sequence $\{\bff(n)\}=\{(f_0(n),\dots,f_{k-1}(n))\}$ in $\RRk$
is uniformly distributed modulo $1$ in $\RRk$ if and only if 
\begin{align}
\label{eq:weylRRk}
\lim_{N\to\infty} \frac1N&\sum_{n=1}^N \exp\left\{
2\pi i (t_0f_0(n)+\dots+t_{k-1}f_{k-1}(n))\right\} =0
\\
\notag
& \qquad \quad ((t_0,\dots,t_{k-1})\in \ZZkz).
\end{align}
\item[(iii)]
A sequence $\{\bff(n)\}=\{(f_0(n),\dots,f_{k-1}(n))\}$ in $\RRk$
is uniformly distributed modulo $1$ in $\RRk$ if and only if, for any 
vector $(t_0,\dots,t_{k-1})\in \ZZkz$, 
the sequence $\{t_0f_0(n)+\dots + t_{k-1}f_{k-1}(n)\}$ is 
uniformly distributed modulo $1$.
\end{itemize}
\end{lem}

Part (iii) of this result is obtained by  combining the $k$-dimensional
Weyl criterion in (ii) with the one-dimensional Weyl criterion in (i),
applied to the sequence $\{t_0f_0(n)+t_1f_1(n)+\dots + t_{k-1}f_{k-1}(n)\}$.

We next state the special case of the $k$-dimensional Weyl's Criterion 
when $\bff$ is of the form
$\bff(n)=(f(n),f(n+1),\dots,f(n+k-1))$, where
$\{f(n)\}$ is a given sequence of real numbers.  
It will be convenient to introduce the notation 
\begin{equation}
\label{eq:ft}
f_{\bft}(n)=\sum_{i=0}^{k-1}t_i f(n+i),
\end{equation}
where $\bft=(t_0,t_1,\dots,t_k)$ is any $k$-dimensional vector.
By applying (iii) of Lemma \ref{lem:weyl} with 
$\bff(n)=(f(n),f(n+1),\dots,f(n+k-1))$, we obtain: 

\begin{cor}[Weyl's Criterion for $\{(f(n),f(n+1),\dots f(n+k-1))\}$]
\label{cor:weyl}
Let $\{f(n)\}$ be a sequence of real numbers, and let $k$ be a positive
real number.  Then $\{(f(n),f(n+1),\dots,f(n+k-1))\}$ is uniformly
distributed modulo $1$ in $\RRk$ if and only if, for each 
$\bft\in\ZZkz$,
the sequence $\{f_\bft(n)\}$ is uniformly distributed
modulo $1$.
\end{cor}

\subsection{Local Benford distribution and uniform distribution modulo $1$}
The following lemma characterizes local Benford distribution in terms of
uniform distribution modulo $1$. 

\begin{lem}[Local Benford distribution and uniform distribution modulo
$1$]
\label{lem:localBenford}
Let $b$ be an integer base $\ge 2$, let $\{a_n\}$ be a sequence of 
positive real numbers, and let $f(n)=\log_b a_n$. 
\begin{itemize}
\item[(i)] The sequence $\{a_n\}$ is Benford distributed with respect to
base $b$ if and only if
the sequence $\{f(n)\}$ is uniformly distributed modulo $1$.

\item[(ii)] Let $k$ be a positive integer. 
The sequence  $\{a_n\}$ is locally Benford distributed of order $k$ with
respect to base $b$ if and only if, for each $\bft\in\ZZkz$, the sequence
$\{f_\bft(n)\}$ is uniformly distributed modulo $1$.

\end{itemize}
\end{lem}

\begin{proof}
Assertion (i) is simply a restatement of the definition
\eqref{eq:defBenford} of a Benford distributed sequence.  

For the proof of (ii), we note that the definition
\eqref{eq:defLocalBenford} of a locally Benford distributed sequence 
sequence $\{a_n\}$ is exactly equivalent to the definition of uniform
distribution modulo $1$ of the $k$-dimensional sequence
$\{(f(n),f(n+1),\dots,f(n+k-1))\}$, with $f(n)=\log_b a_n$.
By Corollary \ref{cor:weyl} this in turn is equivalent to 
the uniform distribution modulo $1$ of all sequences $\{f_\bft(n)\}$, 
with $\bft\in\ZZkz$.
\end{proof}

\subsection{The Key Proposition}
We are now ready to recast our main result, Theorem \ref{thm:main1},
in terms of uniform distribution modulo $1$. 

\begin{prop}
\label{prop:key}
Let $k$ be a positive integer and let $\{f(n)\}$ be a sequence 
belonging to one of   the classes $\cki$ in Definition \ref{def:ck}.
Then:
\begin{itemize}
\item[(i)] For each $k$-dimensional vector 
$\bft\in\ZZkz$, the sequence $\{f_\bft(n)\}$
is uniformly distributed modulo $1$.
\item[(ii)] There exists a $(k+1)$-dimensional 
vector $\bft\in\ZZ^{k+1}\setminus\{0\}$ 
such that the sequence $\{f_\bft(n)\}$ 
is \emph{not} uniformly distributed modulo $1$.
\end{itemize}
\end{prop}

Franklin \cite{franklin1963} established a result of the above type 
in the case when $f(n)$ is a polynomial with irrational leading
coefficient.  It is easy to see that in this case we have
$\Delta^k f(n)=k! a_k$, where $k$ is the degree  of $f(n)$ and $a_k$ is its
leading coefficient, so $f(n)$ trivially belongs to the class $C_{k,0}$.
Thus, Proposition \ref{prop:key} can be viewed as a far-reaching
generalization of Franklin's result.

\subsection{Deduction of Theorem \ref{thm:main1} from Proposition
\ref{prop:key}}

Let $b$ be an integer base $\ge 2$, and let $\{a_n\}$ be a sequence of
positive real numbers satisfying the assumptions of the theorem, so that 
the sequence $f(n)=\log_b a_n$ belongs to one of the
classes $\cki$ in Definition \ref{def:ck} for some $k$.
By Proposition \ref{prop:key} it follows that 
$\{f_\bft(n)\}$ is uniformly distributed modulo $1$
for all $k$-dimensional vectors  $\bft\in\ZZkz$, but not for all 
$(k+1)$-dimensional vectors $\bft\in\ZZ^{k+1}\setminus\{\bfo\}$. 
By Lemma \ref{lem:localBenford} it follows that $\{a_n\}$ is locally Benford
of order $k$, but not of order $k+1$.
This establishes part (i) of the theorem.

To prove part (ii), let $f^{(h)}(n)=\log_b a_n^{(h)}$, so that in
particular $f^{(1)}(n)=f(n)$. 
Then, by the   
definition \eqref{eq:iteratedprod} of the iterated products,  
we have, for any  $h\ge2$,
\begin{align*}
f^{(h)}(n) &=\log_b a_n^{(h)}= \log_b \prod_{m=1}^n a_m^{(h-1)}
\\
&=
\sum_{m=1}^n \log_b a_m^{(h-1)} = \sum_{m=1}^n f^{(h-1)}(m),
\end{align*}
and hence
\[
\Delta f^{(h)}(n)=f^{(h)}(n+1)-f^{(h)}(n)=f^{(h-1)}(n+1).
\]
By iteration we obtain
\[
\Delta^{h-1} f^{(h)}(n)=f^{(1)}(n+h-1)=f(n+h-1)
\]
and hence  
\begin{equation}
\label{eq:Deltafh}
\Delta^{h+k-1}f^{(h)}(n)= \Delta^k f(n+h-1).
\end{equation}
By the assumptions of the theorem, the sequence $\{f(n)\}$, and
hence also the shifted sequence $\{f(n+h-1)\}$,  satisfies 
one of the asymptotic conditions defining the classes $\cki$.
By \eqref{eq:Deltafh}, it follows that the sequence $\{f^{(h)}(n)\}$
satisfies the same condition with $k$ replaced by $k+h-1$, and hence
belongs to one of the classes $C_{k+h-1,i}$.
Applying again Proposition \ref{prop:key} and Lemma
\ref{lem:localBenford} we conclude that $\{a_n^{(h)}\}$ is locally
Benford of order $k+h-1$, but not of any higher order.
\qed


\section{Proof of Proposition \protect \ref{prop:key}}
\label{sec:proof-prop-key}

\subsection{Auxiliary results}
We collect here some known results from the theory of uniform
distribution modulo $1$ that we will need for the proof of Proposition
\ref{prop:key}. 

\begin{lem}
[van der Corput's Difference Theorem ({\cite[Chapter 1, Theorem
3.1]{kuipers}})]
\label{lem:van-der-corput-theorem}
Let $\{f(n)\}$ be a sequence of real numbers such that, for each positive
integer $h$, the sequence $\{f(n+h)-f(n)\}$ is uniformly distributed 
modulo $1$. Then $\{f(n)\}$ is uniformly distributed modulo $1$.
\end{lem}

The following result is a discrete version of a classical exponential sum 
estimate of van der Corput (see, e.g., Theorem 2.2 in Graham and Kolesnik 
\cite{graham-kolesnik}), with the second derivative $f''(x)$ replaced by
its discrete analog, $\Delta^2 f(n)$.  It can be proved by following the
argument in \cite{graham-kolesnik}, using the
(discrete) Kusmin-Landau inequality (see \cite{mordell}) in place of
Theorem 2.1 of \cite{graham-kolesnik} (which is a continuous version of
the Kusmin-Landau inequality).

\begin{lem}
[Discrete van der Corput Lemma]
\label{lem:discrete-van-der-corput}
Let $\{f(n)\}$ be a sequence of real numbers, let $a<b$ be 
positive integers,
and suppose that, for some real numbers $\Lambda>0$ and
$\alpha>1$, 
\begin{equation}
\label{eq:discrete-van-der-corput-hypothesis}
\Lambda \le |\Delta^2 f(n)|\le \alpha \Lambda \quad 
\quad (a\le n<b).
\end{equation}
Then 
\begin{equation}
\label{eq:discrete-van-der-corput}
\left|\sum_{n=a}^{b-1} e^{2\pi i f(n)}\right|
\le C \left(\alpha (b-a)\Lambda^{1/2} + \Lambda^{-1/2}\right),
\end{equation}
where $C$ is an absolute constant.
\end{lem}

\subsection{Uniform distribution of sequences  in $\cki$}
In this subsection we show that sequences $\{f(n)\}$ belonging to one of
the classes $\cki$ are uniformly distributed modulo $1$.  We proceed by
induction on $k$.  The base case, $k=1$, is contained in the following lemma.

\begin{lem}[Uniform distribution of sequences in $C_{1,i}$]
\label{lem:ud-mod1-c1i}
Let $\{f(n)\}$ be a sequence of real numbers belonging to one of the
classes $C_{1,i}$, i.e., satisfying one of the conditions
\begin{align}
\label{eq:class-c1z}
&\lim_{n\to\infty}\Delta f(n)=\theta\quad
\text{for some $\theta\not\in\QQ$,}
\\
\label{eq:class-c1a}
&\lim_{n\to\infty}n^\alpha\Delta f(n)=\lambda\quad
\text{for some $\alpha\in(0,1)$ and $\lambda\not=0$,}
\\
\label{eq:class-c1o}
&\lim_{n\to\infty}n\Delta^2 f(n)=\lambda\quad
\text{for some $\lambda\not=0$.}
\end{align}
Then $\{f(n)\}$ is uniformly distributed modulo $1$.
\end{lem}

\begin{proof} 
The case \eqref{eq:class-c1z} of the lemma is Theorem 3.3 in Chapter 1 of
\cite{kuipers}. For the other two cases, \eqref{eq:class-c1a} and
\eqref{eq:class-c1o}, we will provide proofs as we have not been able to
locate specific references in the literature.

\medskip
\textbf{Case \eqref{eq:class-c1a}.}
Fix $\alpha\in (0,1)$ and consider a sequence $\{f(n)\}$ 
satisfying \eqref{eq:class-c1a}.
We will prove that this sequence is uniformly distributed modulo
$1$ by showing that it satisfies the Weyl Criterion
(see \eqref{eq:weyl} in Lemma \ref{lem:weyl}).

Let $t\in\ZZ\setminus\{0\}$ be given, and let $\lambda\not=0$ be as in
\eqref{eq:class-c1a}.  Without loss of generality, we may assume $t>0$ and
$\lambda>0$.  
Let $K$ be a large, but fixed, constant, and consider intervals of the
form $[N,N+KN^{\alpha})$.   
By \eqref{eq:class-c1a} we have, as $N\to\infty$, 
\begin{align}
\label{eq:cka-case2}
f(N+m)-f(N)&=\sum_{h=0}^{m-1}\Delta f(N+h)
=(1+o(1))\sum_{h=0}^{m-1}\lambda (N+h)^{-\alpha}
\\
\notag
&=(1+o(1))m\lambda N^{-\alpha}
\\
\notag 
&=m\lambda N^{-\alpha} + o(1)
\quad (0\le m<KN^\alpha),
\end{align}
where the convergence implied by the 
notation ``$o(1)$'' is uniform in  $0\le m< KN^{\alpha}$.
It follows that
\begin{align}
\label{eq:cka-case3}
\left|\sum_{N\le n<N+KN^{\alpha}}
e^{2\pi i t f(n)}\right|
&=\left|
\sum_{0\le m< K N^{\alpha}}
e^{2\pi i t (f(N+m)-f(N))}\right|
\\
\notag
&=\left|
\sum_{0\le m< K N^{\alpha}}
\left(
e^{2\pi i t m\lambda N^{-\alpha}}+ o(1)\right)\right|. 
\end{align}
Applying the elementary inequality 
\begin{equation}
\label{eq:cka-case4}
\left|\sum_{m=0}^M e^{2\pi i \theta m}\right|
\le \frac{2}{|e^{2\pi i\theta }-1|}
\le \frac{1}{|\theta|} 
\quad (0<|\theta|\le 1/2),
\end{equation}
we deduce that, for some $N_0$, 
\begin{equation}
\label{eq:cka-case5}
\left|\sum_{N\le n<N+KN^{\alpha}}
e^{2\pi i t f(n)}\right| \le \frac{2}{t\lambda N^{-\alpha}} \quad (N\ge N_0).
\end{equation}
Note that the bound \eqref{eq:cka-case5} represents a saving of a factor
$2/(Kt\lambda)$ over the trivial bound for the
exponential sum on the left.
By splitting the summation range $[N_0,N]$ 
 into subintervals of the form $[N', N'+KN'^{\alpha})$ 
 (where the initial interval may be of shorter length)
 and applying
 \eqref{eq:cka-case5} to each of these subintervals, we obtain 
\begin{equation}
\label{eq:cka-case6}
\limsup_{N\to\infty} \frac1N\left|\sum_{n=1}^N 
e^{2\pi i tf(n)}\right|
\le \frac{2}{Kt\lambda}.
\end{equation}
Since $K$ can be chosen arbitrarily large, 
the limit in \eqref{eq:cka-case6} must be $0$.
Hence the Weyl Criterion 
\eqref{eq:weyl} holds, and the proof of the lemma for the case
\eqref{eq:class-c1a} is complete.

\medskip

\textbf{Case \eqref{eq:class-c1o}.}
Suppose $\{f(n)\}$ is a sequence satisfying \eqref{eq:class-c1o}.
As before we will show that $\{f(n)\}$ is uniformly distributed by showing
that it satisfies the Weyl criterion \eqref{eq:weyl} 
for each $t\in\ZZ\setminus\{0\}$.

Without loss of generality we may assume $t>0$ and $\lambda>0$.
With these simplifications, our assumption \eqref{eq:class-c1o} implies 
\[
\frac{\lambda}{2n}<\Delta^2 f(n)<\frac{2\lambda}{n}
\quad(n\ge N_0)
\]
for some positive integer $N_0$.  Hence we have
\begin{equation}
\label{eq:ud-fn-2}
\frac{t\lambda}{4N}< \Delta^2 (t f(n))< \frac{2 t\lambda}{N}\quad 
(N_0\le N\le n< 2N).
\end{equation}
Thus the sequence $\{t f(n)\}$ satisfies the 
assumption \eqref{eq:discrete-van-der-corput-hypothesis} 
of Lemma \ref{lem:discrete-van-der-corput} on any interval of the form 
$[N,M)$, $N_0\le N<M\le 2N$, 
with the constants 
$\Lambda=t\lambda/(4N)$ and $\alpha=8$. It follows that
\begin{align}
\label{eq:ud-fn-3}
\left|\sum_{n=N}^{M-1} e^{2\pi i t f(n)}\right|
&\le C\left(
8(M-N)\left(\frac{t\lambda}{4N}\right)^{1/2}
+\left(\frac{t\lambda}{4N}\right)^{-1/2}\right)
\\
\notag
&\le C_{t,\lambda} N^{1/2}
\quad (N_0\le N<M\le 2N),
\end{align}
where $C_{t,\lambda}$ is a constant depending only on $t$ and $\lambda$.
The desired relation \eqref{eq:weyl} then follows by  
splitting the summation range into dyadic intervals of the form
$[N',2N')$, along with an interval  $[N',N]$, where $N<2N'$.
\end{proof}

\begin{lem}[Uniform distribution of sequences in $\cki$]
\label{lem:ud-mod1-cki}
Let $\{f(n)\}$ be a sequence of real numbers belonging to one of the
classes $\cki$.  Then $\{f(n)\}$ is uniformly distributed modulo $1$.
\end{lem}

\begin{proof}
We proceed by induction on $k$. The base case, $k=1$, is covered by Lemma
\ref{lem:ud-mod1-c1i}.
For the induction step,
let $k\in\NN$ be given, assume any sequence belonging to one of the classes
$\cki$ is uniformly
distributed modulo $1$,  and let $\{f(n)\}$ be a sequence in 
one of the classes $C_{k+1,i}$. Thus, 
$\{f(n)\}$ satisfies one of the relations 
\begin{align}
\label{eq:class-ckplusz}
&\lim_{n\to\infty}\Delta^{k+1} f(n)=\theta\quad
\text{for some $\theta\not\in\QQ$,}
\\
\label{eq:class-ckplusa}
&\lim_{n\to\infty}n^\alpha\Delta^{k+1} f(n)=\lambda\quad
\text{for some $\alpha\in(0,1)$ and $\lambda\not=0$,}
\\
\label{eq:class-ckpluso}
&\lim_{n\to\infty}n\Delta^{k+2} f(n)=\lambda\quad
\text{for some $\lambda\not=0$.}
\end{align}

By Lemma \ref{lem:van-der-corput-theorem}, to show that $\{f(n)\}$ is
uniformly distributed modulo $1$, it suffices to show that, for each
positive integer $h$, the sequence 
\begin{equation}
\label{eq:delta-h}
\Delta_hf(n)=f(n+h)-f(n).
\end{equation}
is uniformly distributed
modulo $1$.  We will do so by showing that the sequences $\{\Delta_h
f(n)\}$ belong to one of the classes $\cki$ and applying the induction
hypothesis.

We have 
\begin{equation}
\label{eq:delta-k-identity}
\Delta^k(\Delta_h  f(n))
=\Delta^k\left(\sum_{i=0}^{h-1}\Delta f(n+i)\right)
=\sum_{i=0}^{h-1} \Delta^{k+1} f(n+i).
\end{equation}
It follows that if $\{f(n)\}$ satisfies 
one of the relations \eqref{eq:class-ckplusz}--\eqref{eq:class-ckpluso}, 
then $\{\Delta_h f(n)\}$ satisfies the corresponding relation with $k+1$
replaced by $k$ and $\theta$ (resp. $\lambda$) replaced by $h\theta$
(resp. $h\lambda$).  Hence $\{\Delta_h f(n)\}$ belongs to one of the classes
$\cki$, and applying the induction hypothesis we conclude that this sequence
is uniformly distributed modulo $1$.
Hence the sequence $\{f(n)\}$ itself is uniformly distributed modulo $1$, 
as desired.
\end{proof}

\subsection{Proof of Proposition \ref{prop:key}, part (i).}
We proceed by induction on $k$. In the case $k=1$ we have $\bft=(t_0)$ and 
$f_{\bft}(n)=t_0f(n)$, so the assertion reduces to showing that if
a sequence $\{f(n)\}$ belongs to one of the classes $C_{1,i}$, then for any
non-zero integer $t$, the sequence $\{t f(n)\}$ is uniformly distributed
modulo $1$. But this follows from Lemma \ref{lem:ud-mod1-c1i}, applied with
the function $tf(n)$ in place of $f(n)$,  upon noting that $\{tf(n)\}$
belongs to $C_{1,i}$ if and only if $\{f(n)\}$ belongs to $C_{1,i}$.

Now suppose that the assertion holds for some integer $k\ge 1$, i.e., 
suppose that for any sequence $\{f(n)\}$ belonging to one of the classes
$\cki$ and any nonzero $k$-dimensional vector $\bft$, the sequence 
$\{f_{\bft}(n)\}$ is uniformly distributed modulo $1$. 

Let $\{f(n)\}$ be a sequence in $C_{k+1,i}$, and let
$\bft=(t_0,t_1\dots,t_{k})\in \ZZ^{k+1}\setminus\{\mathbf{0}\}$ be given.
We seek to show that the sequence $\{f_\bft(n)\}$ is uniformly distributed
modulo $1$.
We distinguish two cases, according to whether or not the sum $\sum_{i=0}^k t_i$ 
vanishes.

\medskip

Suppose first that $\sum_{i=0}^k t_i\not=0$. Using the identities 
\begin{align*}
\Delta^{k+1} f_\bft(n)&=\sum_{i=0}^kt_i \Delta^{k+1}f(n+i),
\quad \Delta^{k+2} f_\bft(n)=\sum_{i=0}^kt_i \Delta^{k+2}f(n+i),
\end{align*}
we see that if $f$ satisfies one of the relations 
\eqref{eq:class-ckplusz}--\eqref{eq:class-ckpluso}, 
then $f_\bft$ satisfies the same relation
with the constants $\theta$
(resp. $\lambda$) replaced by $\theta\sum_{i=0}^k t_i$
(resp. $\lambda\sum_{i=0}^k t_i$).  Since, by our assumption, $\sum_{i=0}^kt_i$
is a non-zero integer, it follows that the sequence $\{f_\bft(n)\}$ belongs
to one of the classes $\cki$.  Hence, Lemma \ref{lem:ud-mod1-cki}  can be
applied to this sequence and shows that it is uniformly distributed modulo $1$.

\medskip

Now suppose that $\sum_{i=0}^k t_i=0$.
In this case we express $f_\bft(n)$ in terms of the difference function 
$\Delta f(n)$, with a view towards applying the induction hypothesis  
to the latter function. We have
\begin{align}
\label{eq:identity0}
f_\bft(n)&=\sum_{i=0}^kt_if(n+i) =
f(n)\sum_{i=0}^kt_i +\sum_{i=1}^k t_i\sum_{j=0}^{i-1} \Delta f(n+j) 
\\
\notag
&=
\sum_{j=0}^{k-1}\left(\sum_{i=j+1}^k t_i\right) \Delta f(n+j), 
\end{align}
where in the last step we have used our assumption $\sum_{i=0}^k t_i=0$.
Setting
\begin{align}
\label{eq:si-ti}
s_j&=\sum_{i=j+1}^k t_i\quad (j=0,1,\dots,k-1),\quad 
\bfs=(s_0,\dots,s_{k-1}),
\\
\label{eq:g-def}
g(n)&=\Delta f(n)=f(n+1)-f(n),
\end{align}
we can write \eqref{eq:identity0} as 
\begin{equation}
\label{eq:identity1}
f_\bft(n)=g_\bfs(n).
\end{equation}
Thus, to complete the induction step, it suffices to show that the sequence
$\{g_\bfs(n)\}$ is uniformly distributed modulo $1$.

Observe that the linear transformation \eqref{eq:si-ti} 
between the $k$-dimensional vectors 
$(s_0,\dots,s_{k-1})$ and $(t_1,\dots,t_{k})$ is an invertible
transformation on $\ZZk$.
In particular, we have 
$(s_0,\dots,s_{k-1})\not=\bfo$ if and only if 
$(t_1,\dots,t_{k})\not=\bfo$.
Our assumptions $(t_0,t_1,\dots,t_k)\not=\bf0$
and $\sum_{i=0}^k t_i=0$, force $(t_1,\dots,t_k)\not=\bfo$, and by
the above remark it follows that the vector $\bfs$ is a nonzero
$k$-dimensional vector with integer coordinates.

Since $\{f(n)\}$  belongs to one of the classes
$C_{k+1,i}$, the sequence $\{g(n)\}=\{\Delta f(n)\}$ belongs to 
one of the classes $\cki$ (cf. Remark \ref{rem:ck}(3)),
and since, as observed above,  $\bfs\in\ZZkz$,
the sequence $\{g_\bfs(n)\}$ satisfies the assumptions of the proposition
for the case $k$. Thus we can apply the induction hypothesis to conclude 
that the sequence $\{g_\bfs(n)\}$, and hence $\{f_\bft(n)\}$, 
is uniformly distributed modulo $1$.

This completes the proof of Proposition \ref{prop:key}(i).

\subsection{Proof of Proposition \ref{prop:key}, part (ii).}
A routine induction argument shows that 
\begin{equation}
\Delta^kf(n)=\sum_{i=0}^k (-1)^{k-i}\binom{k}{i}f(n+i).
\end{equation}
Thus $\Delta^kf(n)$ is of the form $f_\bft(n)$, where $\bft$ is  
the non-zero $(k+1)$-dimensional vector with components
$t_i=(-1)^{k-i}\binom{k}{i}$, $i=0,1,\dots,k$.
We will show that, if $\{f(n)\}$ belongs to one of the classes 
$\cki$, then $\{\Delta^k f(n)\}$,  and hence $\{f_\bft(n)\}$ with the
above choice of $\bft$,  is not uniformly distributed modulo $1$. 

If $\{f(n)\}$ belongs to the class $\ckz$, then
$\Delta^kf(n)\to \theta$ as $n\to\infty$, 
while if $\{f(n)\}$ belongs to the class $\cka$, then
$\Delta^kf(n)\to 0$ as $n\to\infty$.   Thus in either case
$\{\Delta^k f(n)\}$ cannot be uniformly distributed modulo $1$.

Now suppose $\{f(n)\}$ belongs to the class $\cko$. Then we have, 
for any integers $N\ge 1$ and $0\le m\le N$, 
\begin{align*}
\Delta^{k}f(N+m)-\Delta^{k}f(N)&=\sum_{h=0}^{m-1}\Delta^{k+1}f(N+h)
\\
&=(1+o(1))\sum_{h=0}^{m-1}\frac{\lambda}{N+h},
\end{align*}
as $N\to\infty$, uniformly in $0\le m\le N$.
Setting $\delta=\min(1,1/(4|\lambda|))$, 
it follows that, for sufficiently large $N$ and $0\le m\le \delta N$, 
\begin{equation*}
|\Delta^{k}f(N+m)-\Delta^{k}f(N)|\le \frac13.
\end{equation*}
But this implies that the numbers $\Delta^{k}f(n)$, $N\le n\le(1+\delta)N$,
after reducing modulo $1$, 
cover an interval of length at most $2/3$.
Hence the sequence $\{\Delta^k f(n)\}$ cannot be uniformly distributed
modulo $1$.

This completes the proof of Proposition \ref{prop:key}(ii).

\section{Proof of the corollaries}
\label{sec:proof-corollaries}

We begin with two auxiliary results. The first is a simple result from
the theory of uniform distribution modulo 1; see, for example, Theorem
1.2 in Chapter 1 of \cite{kuipers}.

\begin{lem}
\label{lem:asympudmod1}
Let $\{u_n\}$ and $\{u_n^*\}$ be sequences of real numbers satisfying
\begin{equation}
\notag
u_n-u_n^*\to0\quad (n\to\infty).
\end{equation}
Then $\{u_n\}$ is uniformly distributed modulo $1$ if and only if 
$\{u_n^*\}$ is uniformly distributed modulo $1$.
\end{lem}

\begin{lem}
\label{lem:asympBenford}
Let $\{a_n\}$ and $\{a_n^*\}$
be sequences of positive real numbers satisfying
\begin{equation}
\label{eq:an-asymp-anstar}
a_n\sim a_n^* \quad (n\to\infty).
\end{equation}
Then $\{a_n\}$ is locally Benford distributed of order $k$
with respect to some base $b$
if and only if  $\{a_n^*\}$ is locally Benford distributed of
order $k$ with respect to the same base $b$.
\end{lem}

\begin{proof}
Set $f(n)=\log_b a_n$ and $f^*(n)=\log_b a_n^*$. The hypothesis
\eqref{eq:an-asymp-anstar} is equivalent to 
\begin{equation}
\label{eq:fn_asymp_fn*}
f(n)-f^*(n)\to 0\quad (n\to\infty).
\end{equation}
By Lemma \ref{lem:localBenford}, 
$\{a_n\}$ is locally Benford distributed of order $k$ if and only if, for
each $\bft=(t_0,\dots,t_{k-1})\in\ZZkz$,
the sequence
$\{f_\bft(n)\}=\{t_0f(n)+\dots + t_{k-1}f(n+k-1)\}$
is uniformly
distributed modulo $1$, and an analogous equivalence holds for the
sequences $\{a_n^*\}$ and $\{f^*_\bft(n)\}$. 
Thus, it suffices to show that if one the two sequences
$\{f_\bft(n)\}$ and $\{f^*_\bft(n)\}$ is uniformly distributed modulo
$1$, then so is the other. 

Now, from \eqref{eq:fn_asymp_fn*} we deduce that, 
for any $\bft=(t_0,\dots,t_{k-1})\in\ZZkz$, 
\[
f_\bft(n)-f^*_\bft(n)= \sum_{i=0}^{k-1}t_i(f(n+i)-f^*(n+i))
\to 0\quad (n\to\infty). 
\]
By Lemma \ref{lem:asympudmod1} this implies 
that $\{f_\bft(n)\}$ is uniformly distributed modulo $1$ if and only if
$\{f^*_\bft(n)\}$ is uniformly distributed modulo $1$, as claimed.  
\end{proof}

We are now ready to prove Corollaries \ref{thm:cor1}, \ref{thm:cor2}, and
\ref{thm:cor3}.

\begin{proof}[Proof of Corollary \ref{thm:cor1}]
By Lemma \ref{lem:asympBenford} we may assume that $a_n$ is of the form
$a_n=\lambda n^\gamma e^{cn^\beta}$, for some  $\lambda>0$,
$\gamma\in\RR$, $c>0$, $\beta>0$, with $\beta$ not an integer,  
Then $\log a_n = \log \lambda + \gamma \log n + cn^\beta$, so 
the function $f(n)=\log_b a_n$ is of the form 
$f(x) = c_0 + c_1 \log x + c_2 x^\beta$, for some constants
$c_0,c_1$, and $c_2\not=0$.
Setting $k=\lceil \beta\rceil$ and $\alpha=k-\beta$,  
we then have 
$f^{(k)}(x) = c_1'x^{-k} + c_2' x^{\beta-k}= c_1'x^{-k}+ c_2'
x^{-\alpha}$ for some constants $c_1'$ and $c_2'\not=0$. 
Hence $f(x)$ satisfies 
$x^\alpha f^{(k)}(x)\to c_2'$  as $x\to\infty$. 
By part (4) of Remark \ref{rem:ck}, this implies that
\[
n^\alpha \Delta^k f(n)\to c_2' \quad  (n\to\infty),
\]
i.e., $\{f(n)\}$ belongs to class $\cka$.
Theorem \ref{thm:main1}  then implies that $\{a_n\}$ has maximal 
local Benford order $k$, as claimed.
\end{proof}

\begin{proof}[Proof of Corollary \ref{thm:cor2}]
Applying Lemma \ref{lem:asympBenford} as before, 
we may assume that $a_n= \lambda n^{P(n)} b^{Q(n)}$,
where $P$ and $Q$ are polynomials and $\lambda>0$.
Let $f(n)=\log_b a_n$. Then $f(x)$ is of the form 
$f(x)=c P(x)\log x + Q(x)$, where $c=1/\log b$.

Let $k_P$ and $k_Q$ denote the degrees of the polynomials $P$ and $Q$
respectively, and suppose first that $k_P<k_Q$.
Setting $k=k_Q$, we have 
$(P(x) \log x)^{(k)} \to 0$ and $Q^{(k)}(x)=k!\theta$, where $\theta$ is
the leading coefficient of $Q$.  Hence $f^{(k)}(x)\to k! \theta$, and
therefore also $\Delta^kf(n)\to k!\theta$. If  now $\theta$ is irrational,
then so is $k!\theta$, 
so $\{f(n)\}$ belongs to class $\ckz$, and 
Theorem \ref{thm:main1}  implies that $\{a_n\}$ has maximal 
local Benford order $k$.

Now suppose that $k_P\ge k_Q$. Setting $k=k_P$, we have 
$Q^{(k+1)}(x)=0$, while $(P(x)\log x)^{(k+1)}\sim
c/x$ as $x\to\infty$ for some nonzero constant $c$.  
Hence, $xf^{(k+1)}(x)\to c$, and
therefore also $n\Delta^{k+1}f(n)\to c$.
Thus $\{f(n)\}$ belongs to class $\cko$, and 
by Theorem \ref{thm:main1} we conclude that $\{a_n\}$ has maximal 
local Benford order $k$.
\end{proof}

\begin{proof}[Proof of Corollary \ref{thm:cor3}]
By Corollary \ref{thm:cor2} (see Example \ref{ex:cor2}), the sequence 
of factorials, 
$\{n!\}$, has maximal local Benford order 1. By part (ii) of Theorem
\ref{thm:main1},  this implies that the iterated product 
sequences $\{n!^{(h)}\}$ obtained from this sequence have maximal local
Benford order $h$, as claimed. 
\end{proof}


\section{Proof of Theorem \protect \ref{thm:main2}}
\label{sec:proof-thm-main2}

We will prove Theorem \ref{thm:main2} by reducing the two statements of the
theorem to equivalent statements about uniform distribution modulo $1$ and
applying known results to prove these statements.  We will need the concept
of \emph{complete uniform distribution modulo $1$}, defined as follows
(see, for example, \cite{niederreiter-tichy1985}).

\begin{defn}[Complete uniform distribution modulo $1$]
\label{def:complete-udmod1}
A sequence $\{f(n)\}$ of real numbers is said to be \emph{completely
uniformly distributed modulo $1$} if, for any positive integer $k$, the 
$k$-dimensional sequence $\{(f(n),f(n+1),\dots,f(n+k-1))\}$ is uniformly
distributed modulo $1$ in $\RRk$.
\end{defn}

From the definition of local Benford distribution (see Definition
\ref{def:localBenford}) we immediately obtain  the following 
characterization of sequences with infinite maximal local Benford order in
terms of complete uniform distribution.

\begin{lem}
[Infinite maximal local Benford order and complete uniform distribution
modulo $1$]
\label{lem:local-benford-complete-udmod1}
Let $b$ be an integer base $\ge2$, let $\{a_n\}$ be a sequence of
positive real numbers, and let $f(n)=\log_b a_n$.
Then $\{a_n\}$ has infinite maximal local Benford order with respect 
to base $b$ if and only if the sequence $\{f(n)\}$ is completely uniformly
distributed modulo $1$.
\end{lem}

Now note that for the sequences $\{a_n\}=\{a^{\theta^n}\}$ considered in
Theorem \ref{thm:main2} the function $f(n)=\log_b a_n$ 
has the form
$f(n)=\log_b(a^{\theta^n})=\alpha \theta^n$, where $\alpha=\log_ba$ is a
positive real number.
Thus, in view of Lemmas \ref{lem:localBenford} and
\ref{lem:local-benford-complete-udmod1}, Theorem \ref{thm:main2} reduces 
to the following proposition:

\begin{prop}
\label{prop:key2}
Let $\alpha>0$ be a real number.
\begin{itemize}
\item[(i)]  For almost all irrational numbers $\theta>1$, the sequence 
$\{\alpha \theta^n\}$ is completely uniformly distributed modulo $1$.
\item[(ii)] If $\theta$ is algebraic number of degree $k$, then there exists 
a $(k+1)$-dimensional vector $\bft=(t_0,t_1,\dots,t_k)
\in\ZZ^{k+1}\setminus\{0\}$ for which 
the sequence $\{\sum_{i=0}^kt_i\alpha\theta^{n+i}\}$ is not uniformly distributed modulo $1$.
\end{itemize}
\end{prop}

\begin{proof}[Proof of Proposition \ref{prop:key2}]
The two results are implicit in Franklin \cite{franklin1963}.
Part (i) is a special case of \cite[Theorem 15]{franklin1963}.
Part (ii) is essentially implicit in the proof of \cite[Theorem 16]{franklin1963}
and can also be seen as follows:  
Suppose $\theta$ is algebraic of degree
$k$. Then there exists a polynomial $p(x)=\sum_{i=0}^k a_i x^i$ 
with $a_i\in\ZZ$, $a_k\not=0$, such that $p(\theta)=0$. Letting 
$\bft=(a_0,a_1,\dots,a_k)$, we then have $\bft\in\ZZ^{k+1}\setminus\{0\}$
and
\[
\sum_{i=0}^k t_i \alpha \theta^{n+i}
=\sum_{i=0}^k a_i \alpha \theta^{n+i}
=\alpha\theta^nP(\theta)=0
\]
for all $n\in\NN$.
Thus, the sequence 
$\{\sum_{i=0}^k a_i \alpha \theta^{n+i}\}$
cannot be uniformly distributed modulo $1$. 
\end{proof}


\section{Concluding remarks}
\label{sec:concluding-remarks}

We have chosen the classes $\cki$  as our basis for Theorem \ref{thm:main1}
as these classes have a relatively simple and natural definition, while
being sufficiently broad to cover nearly all of the sequences for which
Benford's Law is known to hold. However, it is clear that similar
results could be proved under a variety of other assumptions on the
asymptotic behavior of $\Delta^k f(n)$.  For example, the class $\cka$
could be generalized to sequences $\{a_n\}$ for which $f(n)=\log_b a_n$
satisfies $n^\alpha(\log n)^\beta\Delta^k f(n)\to \lambda$  for some
constants $\lambda\not=0$, $0<\alpha<1$ and $\beta$.

A natural question is whether asymptotic conditions like those above on
the behavior of $\Delta^k f(n)$ can be replaced by Fejer type monotonicity
conditions as in Theorem 3.4 in Chapter 1 of \cite{kuipers}.  The inductive
argument we have used to prove Proposition \ref{prop:key} depends crucially
on having an asymptotic relation for $\Delta^k f(n)$ and completely breaks
down if we do not have an asymptotic formula of this type available.  
In particular, it is not clear if the conclusion of Proposition
\ref{prop:key} remains valid under the Fejer type conditions of
\cite[Chapter 1, Theorem 3.4]{kuipers}, which require that $\Delta^k f(n)$
be monotone and satisfy $\Delta^k f(n)\to 0$ and $n\Delta^k f(n)\to\infty$
as $n\to\infty$.

An interesting feature of our results, pointed out to the authors by the
referee, is that the quality of the local Benford distribution of a sequence
is largely independent of the quality of its global Benford distribution. A
sequence can have excellent global distribution properties (in the sense that
the leading digit frequencies  converge very rapidly to the Benford
frequencies), while having very poor local distribution properties.  For
instance, any geometric sequence $\{a^n\}$ with $\log_{10}a\not\in\QQ$ has 
maximal local Benford order $1$ in base $10$ and thus  possess the smallest
level of local Benford distribution among sequences that are Benford
distributed. On the other hand, the rates of convergence of the leading
digit frequencies in such sequences are closely tied to the irrationality 
exponent of $\log_{10} a$ and can vary widely.

Our results suggest that the rate of growth of a sequence $\{a_n\}$ is
closely tied to the maximal local Benford order, provided 
$f(n)=\log_b a_n$ behaves, in an appropriate sense, sufficiently
``smoothly''.  In particular, sequences $\{a_n\}$ for which $f(n)$ is a
``smooth'' function of polynomial (or slower) rate of growth cannot be
expected to have  infinite maximal local Benford order.  However,
this heuristic does not apply to sequences for which $f(n)$ behaves more
randomly.   One such example is the sequence of Mersenne numbers,
$2^{p_n}-1$, where $p_n$ is the $n$-th prime. In this case the behavior of
$f(n)$ is determined by the behavior of the sequence of primes, which, 
while growing at a smooth rate (namely, $n\log n$), 
at the local level exhibit  random-like behavior.  Indeed, recent 
numerical evidence (see \cite{mersenne-benford}) suggests that the
sequence of Mersenne numbers, $\{2^{p_n}-1\}$, does have infinite maximal
local Benford order. This is in stark contrast to ``smooth'' sequences with
similar rate of growth such as $\{2^{n\log n}-1\}$, which, by Theorem
\ref{thm:main1}, have maximal local Benford order $1$.

Theorem \ref{thm:main2} could be generalized and strengthened in several
directions by using known metric results on complete uniform distribution. 
For example, Niederreiter and Tichy \cite{niederreiter-tichy1985} showed
that, for any sequence $\{k_n\}$ of distinct positive integers and almost
all $\theta>1$, the sequence $\{\theta^{k_n}\}$ is uniformly distributed
modulo $1$. Their argument applies equally to sequences of the form
$\{\alpha\theta^{k_n}\}$, where $\alpha>0$. By following the proof of
Theorem \ref{thm:main2}, the latter result translates to a statement on the
maximal local Benford order of sequences of the form
$\{a^{\theta^{k_n}}\}$.   

A well-known limitation of metric results of the
above type is that they are not constructive: the results guarantee the
\emph{existence} of sequences with the desired distribution properties, but 
are unable to determine whether a \emph{given} sequence has these properties.
The same limitations apply to the result of Theorem
\ref{thm:main2}. 
Thus, while Theorem \ref{thm:main1} allows us to construct sequences
of arbitrarily large \emph{finite} local Benford order (for example,
sequences of the form $\{2^{n^d}\}$), we do not know of a single
``natural'' example of a sequence with infinite maximal local Benford
order.


\vspace*{4ex}


\begin{thebibliography}{10}

\bibitem{anderson2011}
T.~C. Anderson, L.~Rolen, and R.~Stoehr, \emph{Benford's law for
  coefficients of modular forms and partition functions}, Proc. Amer. Math.
  Soc. \textbf{139} (2011), no.~5, 1533--1541. 

\bibitem{benford}
F.~Benford, \emph{The law of anomalous numbers}, Proc. Amer.
  Philosophical Soc. \textbf{78} (1938), no.~4, 551--572.

\bibitem{berger2011}
A.~Berger and T.~P. Hill, \emph{A basic theory of {B}enford's law},
  Probab. Surv. \textbf{8} (2011), 1--126. 

\bibitem{berger2015}
\bysame, \emph{An introduction to {B}enford's law}, Princeton University Press,
  Princeton, NJ, 2015. 

\bibitem{benfordonline}
A.~Berger, T.~P.~Hill, and E.~Rogers, 
\emph{Benford online bibliography}, 
\url{http://www.benfordonline.net}. Last accessed 06.10.2018.


\bibitem{mersenne-benford}
Z.~Cai, M.~Faust, A.~J. Hildebrand, J.~Li, and Y.~Zhang,
  \emph{Leading digits of {M}ersenne numbers}, Preprint (2018).


\bibitem{diaconis}
P.~Diaconis, \emph{The distribution of leading digits and uniform
  distribution {${\rm mod}$} {$1$}}, Ann. Probability \textbf{5} (1977), 
  no.~1, 72--81. 




\bibitem{franklin1963}
J. Franklin, \emph{Deterministic simulation of random processes}, 
Math. Comp. \textbf{17} (1963), 28--59.


\bibitem{graham-kolesnik}
S.~W.~Graham and G.~Kolesnik, 
\emph{Van der Corput's method of exponential sums}, 
London Mathematical Society Lecture Note Series, 
Cambridge University Press, Cambridge, 1991.

\bibitem{hill1995}
T.~P. Hill, \emph{The significant-digit phenomenon}, Amer. Math. Monthly
  \textbf{102} (1995), no.~4, 322--327. 

\bibitem{kuipers}
L.~Kuipers and H.~Niederreiter, \emph{Uniform distribution of sequences},
  Wiley-Interscience, New York-London-Sydney, 1974.

\bibitem{masse2011}
B.~Mass{{\'e}} and D.~Schneider, \emph{A survey on weighted densities
  and their connection with the first digit phenomenon}, Rocky Mountain J.
  Math. \textbf{41} (2011), no.~5, 1395--1415.

\bibitem{masse2015}
\bysame, \emph{Fast growing sequences of numbers and the first digit
  phenomenon}, Int. J. Number Theory \textbf{11} (2015), no.~3, 705--719.

\bibitem{miller2015}
S.~J. Miller (ed.), \emph{Benford's {L}aw: Theory and Applications},
  Princeton University Press, Princeton, NJ, 2015. 


\bibitem{mordell}
L. J. Mordell, 
\emph{On the Kusmin-Landau inequality for exponential sums}, 
Acta Arith. \textbf{4} (1958), 3--9.

\bibitem{newcomb}
S. Newcomb, \emph{Note on the frequency of use of the different
  digits in natural numbers}, Amer. J. Math. \textbf{4} (1881), no.~1-4,
  39--40.
  

\bibitem{niederreiter-tichy1985}
H. Niederreiter and R. Tichy, 
\emph{Solution of a problem of Knuth on complete uniform
              distribution of sequences}, Mathematika \textbf{1} (1985),
              26--32.  

\bibitem{nigrini2012}
M. Nigrini, 
\emph{Benford's Law: Applications for forensic accounting, auditing, and
fraud detection}, John Wiley Sons, Inc., 2012.

\bibitem{raimi1976}
R.~A. Raimi, \emph{The first digit problem}, Amer. Math. Monthly \textbf{83}
  (1976), no.~7, 521--538. 


\end{thebibliography}
\end{document}